\documentclass[12pt]{amsart}
\usepackage[T1]{fontenc}
\usepackage[pdftex]{graphicx}
\usepackage{url}
\usepackage{hyperref}
\usepackage{mathtools}
\mathtoolsset{showonlyrefs}
\usepackage{anysize}
\usepackage{geometry}
\marginsize{3.0cm}{3.0cm}{3.0cm}{3.0cm}
\usepackage{amssymb}
\theoremstyle{plain}
\numberwithin{equation}{section}
\newtheorem{theorem}{Theorem}[section]
\newtheorem{lemma}[theorem]{Lemma}
\newtheorem{corollary}[theorem]{Corollary}

\newtheorem{remark}[theorem]{Remark}
\newtheorem{definition}[theorem]{Definition}
\newcommand{\ze}{\zeta}

\newcommand{\ii}{\mathsf{i}}
\newcommand{\mo}{\mathsf{mod}~2\pi}
\newcommand{\mop}{\mathsf{mod}~\pi}
\newcommand{\qedw}{\hfill \ensuremath{\Box}}
\begin{document}

\title[Generalized Egerv\'ary's results for harmonic trinomials]{Egerv\'ary's theorems for harmonic trinomials}

\author{Gerardo Barrera}
\address{University of Helsinki, Department of Mathematics and Statistics.
P.O. Box 68, Pietari Kalmin katu 5, FI-00014. Helsinki, Finland.
\url{https://orcid.org/0000-0002-8012-2600}}
\email{gerardo.barreravargas@helsinki.fi}
\thanks{*Corresponding author: Gerardo Barrera.}
\author{Waldemar Barrera}
\address{Facultad de Matem\'aticas, Universidad Aut\'onoma de Yucat\'an. Anillo Perif\'erico Norte Tablaje CAT 13615, M\'erida, Yucat\'an, M\'exico.
\url{https://orcid.org/0000-0001-6885-5556}}
\email{bvargas@correo.uady.mx}
\author{Juan Pablo Navarrete}
\address{Facultad de Matem\'aticas, Universidad Aut\'onoma de Yucat\'an. Anillo Perif\'erico Norte Tablaje CAT 13615, M\'erida, Yucat\'an, M\'exico.
\url{https://orcid.org/0000-0002-3930-4365}}
\email{jp.navarrete@correo.uady.mx}


\begin{abstract}
In this manuscript, 
we study the arrangements of
the roots in the complex plane for the lacunary harmonic polynomials called harmonic trinomials.
We provide necessary and sufficient conditions so that two general harmonic trinomials have the same set of roots up to a rotation around the origin in the complex plane, a reflection over the real axis, or a composition of the previous both transformations.
This extends  the results of Jen\H{o} Egerv\'ary given in~\cite{eger} for the setting of trinomials to the setting of harmonic trinomials.
\end{abstract}

\maketitle

\section{\textbf{Introduction, main results and their consequences}}\label{sec:intro}
\subsection{\textbf{Introduction}}
\hfill

\noindent
The computation and the quantitative location of the roots for polynomials are important in many research areas, and therefore a vast literature in both pure mathematics and applied mathematics has been produced, we refer 
to~ \cite{Borweinbook,Mardenbook,Pan1997,Prasolovbook,RahmanSchmeisser,SheilSmallbook} and the references therein.

Given two positive integers $m$ and $n$, a \textit{trinomial} of degree $n+m$ is a lacunary polynomial with three terms of the form
\begin{equation}\label{def:TRI}
T(z):=Az^{n+m} + B z^m + C\quad \textrm{ for all }\quad z\in \mathbb{C},
\end{equation}
where $A$, $B$ and $C$ are non-zero complex numbers.
Despite the apparent simplicity 
of~\eqref{def:TRI},
the well-known works of P. Ruffini, N.H. Abel and \'E. Galois imply that for $n+m\geq 5$ and generic trinomials of the form~\eqref{def:TRI} there is no formula for their roots in terms of the so-called radicals.
For the literature reporting geometric, topological, quantitative and qualitative behavior of the roots for trinomials of the form~\eqref{def:TRI} we refer 
to~\cite{Barrera2,BarreraMaganaNavarrete,Belki,Bohl,Botta,BRILLESLYPER,BrilleslyperSchaubroeck2018,BrilleslyperSchaubroeck2014,Cermak2015survey,CermakFedorkova2022,Cermak2022,Cermak2015,Cermak2019,eger,Fell,Howell,Kipnis2004,Kuruklis,Matsunaga2010,Melman,Nekrassoff,Szabo,Theobald} and the references therein. 

In~\cite{eger} J.~Egerv\'ary analyzes the roots of general trinomials. More precisely, he studied
\begin{itemize}
\item[(I)] the arrangements of
the roots of~\eqref{def:TRI} in the complex plane, that is, provides necessary and sufficient conditions so that
two general \textit{trinomials} possess the same set of roots up to a rotation around the origin in the complex plane, a reflection over the real axis, or a composition of the previous both transformations.
The latter is an equivalence relation,
which in the sequel we refer to as its Egerv\'ary equivalent,
see Definition~\ref{def:eger} below.
\item[(II)] the description of geometric sectors for the localization of the roots of~\eqref{def:TRI}.
\end{itemize}
Since~\cite{eger} is written in Hungarian, many of the results given there have been rediscovered afterwards.
We refer to~\cite{Szabo} for an English review of~\cite{eger}.

In this manuscript, we extend~(I) to the setting of harmonic trinomials. More precisely, to the setting of lacunary harmonic polynomials of the form
\begin{equation}\label{eq:Hdef}
H(z):=Az^{n+m} + B \overline{z}^m + C\quad \textrm{ for all }\quad z\in \mathbb{C},
\end{equation}
where $A$, $B$ and $C$ are non-zero complex numbers,
and  $\overline{\zeta}$ denotes the complex conjugate of
the given complex number $\zeta$, see Theorem~\ref{th:egervaryI} below.
As a consequence of Theorem~\ref{th:egervaryI} we obtain the following results.
\begin{itemize}
\item[(a)]
A characterization of the class of harmonic trinomials of the form~\eqref{eq:Hdef} which are Egerv\'ary equivalent with a harmonic trinomial with real coefficients,
see Corollary~\ref{cor:real} below.
\item[(b)] A harmonic trinomial of the form~\eqref{eq:Hdef} with different roots having the same complex modulus is Egerv\'ary equivalent to a harmonic trinomial with real coefficients, see Corollary~\ref{cor:difroots} below.
\item[(c)]
 A harmonic trinomial of the form~\eqref{eq:Hdef} with a root of multiplicity at least two is Egerv\'ary equivalent to a harmonic trinomial with real coefficients,
see Corollary~\ref{cor:leasttwo} below.
\end{itemize}
In addition, Theorem~\ref{th:egervaryI} with the help of the following results
in~\cite{Barrera}, Lemma~2.6, Lemma~2.11, Lemma~A.3 and Proposition~2.3, 
 yields the following statements.
\begin{itemize}
\item[(e)] A geometric degenerate triangle condition on the modulus of the coefficients 
of~\eqref{eq:Hdef} so that~\eqref{eq:Hdef} is Egerv\'ary equivalent to a harmonic trinomial with real coefficients, see Theorem~\ref{lem:geometric} below.
\item[(f)] Two harmonic trinomials of the form~\eqref{eq:Hdef} with roots having the same complex modulus (such roots may be different) and satisfying that the ratio between the complex modulus of their respective coefficients with the same degree is constant, are Egerv\'ary equivalent, 
see Theorem~\ref{th:samemodulus} below.
\end{itemize}

By Bezout's Theorem (Theorem~1 and Theorem~5 in~\cite{Wilmshurst}) it follows that~\eqref{eq:Hdef} has at most $(n+m)^2$ roots. 
Recently, in Corollary~1.4 of~\cite{Barrera}, it is shown that~\eqref{eq:Hdef} has at most $n+3m$ roots. Moreover, such bound is sharp in the sense that there exist harmonic trinomials with exactly $n+3m$ roots.
In general, there exist harmonic polynomials with exactly $(n+m)^2$ roots, see for instance Section~2 in~\cite{Wilmshurst} or p.~2080 
of~\cite{Bshouty}.

Recently, the corresponding geometric sectors as in~(II) for harmonic trinomials of the 
form~\eqref{eq:Hdef} with $A=1, B\in \mathbb{C}\setminus\{0\}, C=-1$ has been derived in~\cite{Gao}. For references about location, counting, geometry, and lower/uppers bounds for the moduli of roots for 
harmonic polynomials including probabilistic approaches and numerical experiments, we refer 
to~\cite{Barrera,Brooks,Bshouty,BshoutyLyzzaik,Gao,Geleta,GeletaAlemu,Geyer,Harrat,
Hauenstein,Khavinson,KhavinsonNeumann2008,KhavinsonNeumann,KhavinsonSwiatek,Lee,LeeSaez,
Lerario,LundbergThomack,LiWei,Lundberg,LundbergThomack,Thomackpreprint,ThomackTyree,Wilmshurst} and the references therein.

\subsection{\textbf{Preliminaries and main results}}\label{sub:preli}
\hfill

\noindent
In this subsection, we present the preliminaries and state the results of this manuscript.

Given $m,n\in \mathbb{N}:=\{1,2,\ldots,\}$, we consider two harmonic trinomials 
\begin{equation}\label{eq:f1}
h_1(z):= A_1 z^{n+m}+B_1\overline{z}^m +C_1 \quad \textrm{ for all }\quad z\in \mathbb{C},
\end{equation}
and
\begin{equation}\label{eq:f2}
h_2(z):= A_2 z^{n+m}+B_2\overline{z}^m +C_2 \quad \textrm{ for all }\quad z\in \mathbb{C},
\end{equation}
where $A_1$, $A_2$, $B_1$, $B_2$, $C_1$, $C_2$ are non-zero complex numbers.

We start with the following definition, which rigorously encodes the arrangements of roots that are equivalent.

\begin{definition}[Egerv\'ary equivalent]\label{def:eger}
\hfill

\noindent
Let $h_1$ and $h_2$ be the harmonic polynomials given in~\eqref{eq:f1} 
and~\eqref{eq:f2}, respectively.
We say $h_1$ and $h_2$ are Egerv\'ary equivalent
if and only if the set of roots of $h_1$ differs of the set of roots of $h_2$ by 
\begin{itemize}
\item[(a)] a rotation around the origin in the complex plane,
\item[(b)] a reflection over the real axis,
\item[(c)] a composition of both transformations given in~(a) and~(b).
\end{itemize}
More precisely,
there exist a non-zero complex number $c$ and a real number  $\delta$ satisfying
\begin{equation}\label{eq:equi}
h_1(z)=c h_2(e^{\ii \delta}z)\quad \textrm{ for all }\quad z\in \mathbb{C}
\end{equation}
or
\begin{equation}\label{eq:equi1}
h_1(z)= c\overline{h}_2(e^{\ii \delta}z)\quad \textrm{ for all }\quad z\in \mathbb{C},
\end{equation} 
where  $\ii$ denotes the imaginary unit and
\[
\overline{h}_2(z):= \overline{A_2}z^{n+m}+\overline{B_2}\overline{z}^m +\overline{C_2}\quad \textrm{ for all }\quad  z\in \mathbb{C}.
\]
\end{definition}
We note that Definition~\ref{def:eger} defines an equivalence relation. 
In addition, we observe that~\eqref{eq:equi} and~\eqref{eq:equi1} are mutually exclusive whenever some of the coefficients $A_2$, $B_2$ or $C_2$ is not a real number. Indeed,
if~\eqref{eq:equi} and~\eqref{eq:equi1} both hold true, we have 
\[h_1(z)=c h_2(e^{\ii \delta}z)=c\overline{h}_2(e^{\ii \delta}z)\quad \textrm{ for all }\quad z\in \mathbb{C},
\]
which yields
\[
(A_2-\overline{A_2}) z^{n+m}+(B_2-\overline{B_2})\overline{z}^m +(C_2-\overline{C_2})=0\quad \textrm{ for all }\quad  z\in \mathbb{C}.
\]
The latter implies $A_2$, $B_2$ and $C_2$ are real numbers.

Along this manuscript, $|\ze|$ denotes the complex modulus of the given complex number $\ze$. 
We recall that the polar representation of $\zeta$ is given by $\zeta=|\zeta|e^{\ii \varphi}$, where $\varphi\in [0,2\pi)$ is the argument of $\zeta$.
Moreover,
for any real numbers $x$ and $y$, we write that 
\[
x\equiv y\quad \mo\quad \textrm{ if and only if }\quad x-y=2k\pi\quad \textrm{ for some }\quad k\in \mathbb{Z}.
\]

The first main result of this manuscript is the following extension of the results given in Equation~(2) and Equation~(3) in p.~37 
of~\cite{eger} or Theorem~1 in the 
survey~\cite{Szabo},  to the setting of harmonic trinomials. It reads as follows.

\begin{theorem}[Egerv\'ary's Theorem for harmonic trinomials]\label{th:egervaryI}
\hfill

\noindent
Let $h_1$ and $h_2$ be the harmonic polynomials given in~\eqref{eq:f1} 
and~\eqref{eq:f2}, respectively.
Then the following holds true:
$h_1$ and $h_2$ are  Egerv\'ary equivalent if and only if
\begin{equation}\label{eq:cociente}
\left|\frac{A_1}{A_2}\right|= \left|\frac{B_1}{B_2}\right|= \left|\frac{C_1}{C_2}\right|
\end{equation}
and  
\begin{equation}\label{eq:mod}
m(\alpha_1\pm \alpha_2)+(n+m)(\beta_1\pm \beta_2) - (n+2m)(\gamma_1\pm \gamma_2)\equiv 0\quad \mo,
\end{equation}
where $\alpha_1$, $\alpha_2$, $\beta_1$, $\beta_2$, $\gamma_1$ and $\gamma_2$ are the arguments in the polar representation of $A_1$, $A_2$, $B_1$, $B_2$, $C_1$ and $C_2$, respectively.
\end{theorem}
The proof is given in 
Subsection~\ref{sub:th:egervaryI}.

\begin{remark}[About the choice of $\pm$]
\hfill

\noindent
We point out that~\eqref{eq:mod} reads
\[
m(\alpha_1+ \alpha_2)+(n+m)(\beta_1+ \beta_2) - (n+2m)(\gamma_1+ \gamma_2)\equiv 0\quad \mo
\]
or
\[
m(\alpha_1- \alpha_2)+(n+m)(\beta_1- \beta_2) - (n+2m)(\gamma_1- \gamma_2)\equiv 0\quad \mo.
\]
\end{remark}

The following corollary is the analogous of the results given in Equation~(1), 
Equation~(5) and Equation~(6) in p.~38 of~\cite{eger} or p.~100 in~\cite{Szabo}  to the setting of harmonic trinomials. It reads as follows.

\begin{corollary}[The class of harmonic trinomials with real coefficients]\label{cor:real}
\hfill

\noindent
Let $h(z)=Az^{n+m}+B\overline{z}^m +C$ for all $z\in \mathbb{C}$ be a harmonic trinomial whose coefficients $A$, $B$ and $C$ are non-zero complex numbers.
We consider the polar representation of $A$, $
B$ and $C$, that is, $A= |A|e^{\ii \alpha}$, $B= |B|e^{\ii \beta}$, $C= |C|e^{\ii \gamma}$.  Then the following statements are equivalent.
\begin{itemize}
\item[(i)] The harmonic trinomial $h$ is Egerv\'ary equivalent  to a harmonic trinomial with real coefficients.
\item[(ii)] 
The angular relation 
\begin{equation}\label{eq:reales}
m\alpha+(n+m)\beta - (n+2m)\gamma\equiv 0 \quad \mop
\end{equation}
holds true.
\item[(iii)] The coefficients $A$, $B$ and $C$ satisfy that
\begin{equation}
\frac{A^mB^{n+m}}{C^{n+2m}}\quad \textrm {is a real number}.
\end{equation}
\end{itemize}
\end{corollary}
The proof is provided in 
Subsection~\ref{sub:cor:real}.

The following corollary is the analogous of Statement~III in p.~40 of~\cite{eger}.
\begin{corollary}[Different roots with the equal modulus]\label{cor:difroots}
\hfill

\noindent
Let $h(z)=Az^{n+m}+B\overline{z}^m +C$ for all $z\in \mathbb{C}$ be a harmonic trinomial whose coefficients $A$, $B$ and $C$ are non-zero complex numbers.
We consider the polar representation of $A$, $
B$ and $C$, that is, $A= |A|e^{\ii \alpha}$, $B= |B|e^{\ii \beta}$, $C= |C|e^{\ii \gamma}$. Assume that $h$ has at least two different roots with the same modulus. Then $h$ is Egerv\'ary equivalent to a harmonic trinomial with real coefficients.
\end{corollary}

The proof is presented in Subsection~\ref{sec:cor:difroots}.

In the sequel, we remark that the converse of the Statement~III in p.~40 of~\cite{eger} does not hold true in general.

\begin{remark}[Converse of Corollary~\ref{cor:difroots} for degree three or more]
\hfill

\noindent
Let $h(z)=z^2+\frac{\sqrt{3}}{6}\overline{z}-\frac{1}{4}$ for all $z\in \mathbb{C}$.
By Corollary~1.4 in~\cite{Barrera} we have that $h$ has at most four different roots in $\mathbb{C}$. In fact, a straightforward computation yields that $h$ has only two roots, which are given by
\[
z_1:=\frac{-1-\sqrt{13}}{2\sqrt{12}}\approx -0.664\quad \textrm{ and } \quad  z_2:=\frac{-1+\sqrt{13}}{2\sqrt{12}}\approx 0.376.
\]
Since  $|z_1|\neq |z_2|$, we have that the converse of Corollary~\ref{cor:difroots} is not valid when $n+m=2$, i.e., $n=m=1$.

For $n+m\in \mathbb{N}\setminus\{1,2\}$ we claim that all the roots of  $h(z)=Az^{n+m}+B\overline{z}^m +C$ for all $z\in \mathbb{C}$, where $A,B,C\in \mathbb{R}\setminus\{0\}$ cannot be real numbers. Indeed, by Descartes' rule of signs we have that $h$ has at most two real roots and hence $h$ has at least one complex root $\zeta$. It is not hard to see that $\overline{\zeta}$ is also a root of $h$ and hence the converse of Corollary~\ref{cor:difroots} holds true.
\end{remark}

The following corollary is the analogous of Statement~IV in p.~40 of~\cite{eger}.

\begin{corollary}[Root with multiplicity at least two]\label{cor:leasttwo}
\hfill

\noindent
Let $h(z)=Az^{n+m}+B\overline{z}^m +C$ for all $z\in \mathbb{C}$ be a harmonic trinomial whose coefficients $A$, $B$ and $C$ are non-zero complex numbers.
We consider the polar representation of $A$, $
B$ and $C$, that is, $A= |A|e^{\ii \alpha}$, $B= |B|e^{\ii \beta}$, $C= |C|e^{\ii \gamma}$.  Assume that $h$ has a root of multiplicity at least two with modulus $r$. Then $h$ is Egerv\'ary equivalent to a harmonic trinomial with real coefficients. Moreover, 
\[\frac{A^mB^{n+m}}{C^{n+2m}} = \frac{(-1)^{n+m}}{r^{2m(n+m)}}\frac{m^m(n+m)^{n+m}}{n^{n+2m}}.
\]
\end{corollary}

The proof is provided in Subsection~\ref{sec:cor:leasttwo}.

The analogous of the following theorems are not given in~\cite{eger}. 
We state them here since they are interesting on their own.
They are deduced 
using 
Theorem~\ref{th:egervaryI} together with 
the following results
in~\cite{Barrera}, Lemma~2.6, Lemma~2.11, Lemma~A.3 and Proposition~2.3.
\begin{theorem}[Geometric degenerate condition]\label{lem:geometric}
\hfill

\noindent
Let $h(z)=Az^{n+m}+B\overline{z}^m +C$ for all $z\in \mathbb{C}$ be a harmonic trinomial whose coefficients $A$, $B$ and $C$ are non-zero complex numbers.
We consider the polar representation of $A$, $
B$ and $C$, that is, $A=|A|e^{\ii \alpha}$, $B= |B|e^{\ii \beta}$, $C= |C|e^{\ii \gamma}$. Assume that  there exists a root of $h$ with modulus $r$ such that  
$|A|r^{n+m}$, $|B|r^{m}$ and $|C|$ are the side lengths of some degenerate triangle. Then $h$ is Egerv\'ary equivalent to a harmonic trinomial  with real coefficients of the form
\begin{equation}\label{eq:guv}
g_{u,v}(z):= u|A|z^{n+m}+v|B|\overline{z}^m+|C|,\quad  z\in \mathbb{C},
\end{equation} for some $u,v\in \{-1,1\}$.
Moreover, if $n$ and $m$ are co-prime numbers, then 
\begin{itemize}
\item[(a)] for $|C|= |A|r^{n+m}+|B|r^m$ it follows that $h$ is Egerv\'ary equivalent to 
\[
g(z):= |A|z^{n+m}+|B|\overline{z}^m -|C|, \quad z\in \mathbb{C},
\]
\item[(b)] for $|A|r^{n+m}=|B|r^m+|C|$, it follows that $h$ is Egerv\'ary equivalent to 
\[
g(z):= |A|z^{n+m}-|B|\overline{z}^m-|C|, \quad z\in \mathbb{C},
\]
\item[(c)] for $|B|r^{m}= |A|r^{m+n}+|C|$, it follows that $h$ is Egerv\'ary equivalent to 
\[
g(z):=|A|z^{n+m}-|B|\overline{z}^m+|C|,\quad z\in \mathbb{C}.
\]
\end{itemize}
\end{theorem}
The proof is given in Subsection~\ref{sec:lem:geometric}.

\begin{theorem}[Common root with the same modulus]
\label{th:samemodulus}
\hfill

\noindent
Let $h_1$ and $h_2$ be the harmonic polynomials given in~\eqref{eq:f1} 
and~\eqref{eq:f2}, respectively. Assume 
that~\eqref{eq:cociente} holds true.
In addition, assume that there exists $\zeta_1$ and $\zeta_2$ roots of 
$h_1$ and $h_2$, respectively, and satisfying $|\zeta_1|=|\zeta_2|$. Then it follows that
\begin{equation}\label{ecn:one}
m(\alpha_1\pm \alpha_2)+(n+m)(\beta_1\pm \beta_2) - (n+2m)(\gamma_1\pm \gamma_2)\equiv 0 \quad \mo.
\end{equation}
In particular,  $h_1$ and $h_2$ are  Egerv\'ary equivalent.
\end{theorem}
The proof is provided in Subsection~\ref{sec:th:samemodulus}.

The rest of the manuscript is organized as follows. In Section~\ref{sec:proofs}
we provide the proofs of the results given in Section~\ref{sec:intro}. More precisely, in Subsection~\ref{sub:th:egervaryI} we show Theorem~\ref{th:egervaryI}, in Subsection~
\ref{sub:cor:real} we provide the proof of
Corollary~\ref{cor:real}, in Subsection~\ref{sec:cor:difroots} we give the proof of
Corollary~\ref{cor:difroots}, in Subsection~\ref{sec:cor:leasttwo}
we provide the proof of Corollary~\ref{cor:leasttwo}, in Subsection~\ref{sec:lem:geometric} we show Theorem~\ref{lem:geometric} and in Subsection~\ref{sec:th:samemodulus} we prove 
Theorem~\ref{th:samemodulus}. Finally, in Appendix~\ref{append} we state auxiliary results that have been used throughout the manuscript.

\section{\textbf{Proofs of the results}}\label{sec:proofs}

In this section, we present the proofs of the results stated in Subsection~\ref{sub:preli}.

\subsection{\textbf{Proof of Theorem~\ref{th:egervaryI}}}\label{sub:th:egervaryI}
\hfill
\smallskip

\noindent
Assume that $h_1$ and $h_2$ are Egerv\'ary equivalent. By Definition~\ref{def:eger} it is not hard to see that there exist a non-zero complex number $c$ and a real number $\delta$ satisfying
\begin{equation}\label{ec:uno}
h_1(e^{-\ii \delta}z)=ch_2(z)\quad \textrm{ for all }\quad z\in \mathbb{C}
\end{equation}
or 
\begin{equation}\label{ec:dos}
\overline{h}_1(e^{-\ii \delta}z)=\overline{c} h_2(z)\quad \textrm{ for all }\quad z\in \mathbb{C}.
\end{equation}
By~\eqref{ec:uno} we have 
\begin{equation}\label{eq:cuenta}
\frac{A_1}{A_2}= \frac{B_1}{B_2} e^{\ii (n+2m)\delta}= \frac{C_1}{C_2}e^{\ii (n+m)\delta},
\end{equation}
which easily implies~\eqref{eq:cociente}.
By~\eqref{eq:cuenta} we have
\[
\alpha_1-\alpha_2 \equiv \beta_1-\beta_2 +(2m+n)\delta \equiv \gamma_1-\gamma_2 +(n+m)\delta  \quad\mo.
\]
In particular, we obtain
\begin{equation}\label{eq:doscong}
\begin{split}
\gamma_1-\gamma_2 &\equiv \beta_1-\beta_2 +m\delta \quad \mo  \quad \textrm{ and }\\
\alpha_1-\alpha_2 &\equiv \beta_1-\beta_2 +(2m+n)\delta \quad \mo.
\end{split}
\end{equation}
By~\eqref{eq:doscong} we have 
\begin{equation}\label{eq:uno}
\begin{split}
&\hspace{1cm}
m(\alpha_1- \alpha_2)+(n+m)(\beta_1- \beta_2) - (n+2m)(\gamma_1- \gamma_2)\\
&\quad
\hspace{1cm}
\equiv 
m(\beta_1-\beta_2 +(2m+n)\delta)+(n+m)(\beta_1- \beta_2) - (n+2m)(\beta_1-\beta_2 +m\delta)\\
&\quad
\hspace{1cm}
\equiv 0
 \quad 
\mo.
\end{split}
\end{equation}
Analogously,~\eqref{ec:dos} implies
\begin{equation}\label{eq:dos}
\begin{split}
m(\alpha_1+ \alpha_2)+(n+m)(\beta_1+ \beta_2) - (n+2m)(\gamma_1+ \gamma_2)\equiv  0
 \quad \mo.
\end{split}
\end{equation}
By~\eqref{eq:uno} and~\eqref{eq:dos} 
we deduce~\eqref{eq:mod}.

In the sequel, we assume 
that~\eqref{eq:cociente} and~\eqref{eq:mod} are valid. In particular, we have
\begin{equation}\label{eq:la}
\frac{|A_1|}{|A_2|}=\frac{|B_1|}{|B_2|}=\frac{|C_1|}{|C_2|}=:r
\end{equation}
and
\begin{equation}\label{eq:angular}
m(\alpha_1- \alpha_2)+(n+m)(\beta_1-\beta_2)- (n+2m)(\gamma_1- \gamma_2)\equiv 0\quad \mo.
\end{equation}
Without loss of generality, we can assume that 
$r=1$. Since $r=1$,~\eqref{eq:la} 
 implies the existence of real numbers $\theta_1$, $\theta_2$ and $\theta_3$ such that 
\begin{equation}\label{eq:A1A2}
A_1= e^{\ii \theta_1} A_2, \quad  B_1= e^{\ii \theta_2}  B_2 \quad  \textrm{ and } \quad C_1= e^{\ii \theta_3 } C_2.
\end{equation}
Then we have
\begin{equation}\label{eq:angulos}
\alpha_1-\alpha_2 \equiv \theta_1 \quad \mo, \quad 
\beta_1-\beta_2 \equiv \theta_2 \quad \mo \quad \textrm{ and } \quad 
\gamma_1-\gamma_2 \equiv \theta_3 \quad \mo. 
\end{equation}
By~\eqref{eq:angular} and~\eqref{eq:angulos} we obtain
\begin{equation}\label{diofanto1}
m(\theta_1-\theta_3)+(n+m)(\theta_2-\theta_3)
\equiv 0 \quad \mo.
\end{equation}
By Lemma~\ref{lem:parame} in Appendix~\ref{append} we have that the solutions 
of~\eqref{diofanto1} are parametrized   as follows
\begin{equation}\label{eq:thetas}
\theta_1 \equiv \theta_3 +(n+m)\delta \quad  \mo\quad \textrm{ and }\quad  \theta_2\equiv \theta_3 - m\delta \quad  \mo
\end{equation}
for some $\delta\in \mathbb{R}$. 
By~\eqref{eq:A1A2} and~\eqref{eq:thetas} we obtain
\begin{equation}
\begin{split}
h_1(z)&= A_1z^{n+m}+B_1\overline{z}^m +C_1
= A_2e^{\ii \theta_1}z^{n+m}+B_2e^{\ii\theta_2}\overline{z}^m +C_2e^{\ii \theta_3}\\
&=e^{\ii \theta_3}(A_2 e^{\ii (n+m)\delta}z^{n+m}+B_2e^{-\ii m\delta}\overline{z}^m+C_2)
=e^{\ii \theta_3}f_2(e^{\ii \delta}z)
\end{split}
\end{equation}
for all $z\in \mathbb{C}$.
In the case of
\[
m(\alpha_1+ \alpha_2)+(n+m)(\beta_1+\beta_2)- (n+2m)(\gamma_1+ \gamma_2)\equiv 0 \quad \mo
\]
the proof is analogous and we omit it.
\qedw

\subsection{\textbf{Proof of Corollary~\ref{cor:real}}}\label{sub:cor:real}
\hfill
\smallskip

\noindent
We start proving that~(i) implies~(ii).
Since~(i) holds true,  $h$ is Egerv\'ary equivalent  to a harmonic trinomial with real coefficients $g$.
We write $g(z)=A_1z^{n+m}+B_1\overline{z}^m+C_1$, $z\in \mathbb{C}$, where $A_1$, $B_1$ and $C_1$ are real numbers. 
In particular, 
\begin{equation}\label{eq:a1b1r1}
\alpha_1 \equiv 0\quad \mop,\quad \beta_1 \equiv 0\quad \mop\quad \textrm{ and }\quad 
\gamma_1 \equiv 0\quad \mop,
\end{equation} where  $\alpha_1$, $\beta_1$ and $\gamma_1$ are the arguments $A_1$, $B_1$ and $C_1$, respectively. By Theorem~\ref{th:egervaryI} (applied to $h$ and $g$) we have 
\begin{equation}
\begin{split}
m(\alpha\pm \alpha_1)+(n+m)(\beta\pm \beta_1)-(n+2m)(\gamma\pm \gamma_1)\equiv 0\quad \mo,
\end{split}
\end{equation}
which implies
\begin{equation}\label{ecnp}
m\alpha +(n+m)\beta -(n+2m)\gamma \equiv \mp m\alpha_1 \mp (n+m)\beta_1\pm (n+2m)\gamma_1 \quad \mo.
\end{equation}
By~\eqref{eq:a1b1r1}
for any $\ell_1,\ell_2,\ell_3\in \{-1,1\}$ we obtain
\begin{equation}\label{ec:19}
\ell_1 m\alpha_1 +\ell_2 (n+m)\beta_1+\ell_3 (n+2m)\gamma_1
\equiv 0 \quad \mop.
\end{equation}
Hence~\eqref{ecnp} with the help 
of~\eqref{ec:19} yields~\eqref{eq:reales}. 
The proof of~(i) implies~(ii) is finished.

Now, we show that~(ii) implies~(i).
Since~\eqref{eq:reales} holds true,
 we have
\begin{equation}\label{eq:reales1}
m(\alpha-\gamma)+(n+m)(\beta -\gamma)\equiv 0 \quad \mop.
\end{equation}
By Lemma~\ref{lem:paramedos} in Appendix~\ref{append} we obtain
\begin{equation}\label{eq:deltar}
\alpha\equiv \gamma +(n+m)\delta\quad \mop\quad \textrm{ and }  \quad\beta\equiv \gamma -m\delta\quad\mop
\end{equation}
for some $\delta\in \mathbb{R}$.
Then for any $z\in \mathbb{C} $ we obtain
\begin{equation}\label{eq:hequ}
\begin{split}
h(z)&=|A|e^{\ii \alpha}z^{n+m}+|B|
e^{\ii \beta}\overline{z}^m+|C|e^{\ii \gamma}\\
&= e^{\ii\gamma}(|A|e^{\ii(\alpha-\gamma)}z^{n+m}+|B|e^{\ii(\beta-\gamma)}\overline{z}^m+|C|)\\
&= e^{\ii\gamma}(u|A|e^{\ii (n+m)\delta}z^{n+m}+v|B|e^{-\ii m\delta}\overline{z}^m+|C|)\\
&=e^{\ii\gamma}(u|A|(e^{\ii\delta}z)^{n+m}+v|B|(\overline{e^{\ii\delta}z})^m+|C|),
\end{split}
\end{equation}
where $u,v\in \{-1,1\}$.
By~\eqref{eq:hequ} we deduce that $h$ is Egerv\'ary equivalent to the harmonic trinomial
\begin{equation}\label{formareal}
g(z):=u|A|z^{n+m}+v|B|\overline{z}^m+|C|\quad \textrm{ for all }\quad z\in \mathbb{C}.
\end{equation}
The proof of~(ii) implies~(i) is complete.

Finally, the equivalence of~(ii) and~(iii) is straightforward due to the following relation
\begin{equation}
\frac{A^mB^{n+m}}{C^{n+2m}} =\frac{ |A|^m |B|^{n+m}}{ |C|^{n+2m}} e^{\ii(m\alpha +(n+m)\beta-(n+2m)\gamma)}.
\end{equation}
\qedw

\subsection{\textbf{Proof of Corollary~\ref{cor:difroots}}}\label{sec:cor:difroots}
\hfill
\smallskip

\noindent
Assume that $h$ has two different roots with modulus $r>0$.
After a rotation, one can see that $h$ is Egerv\'ary equivalent to a harmonic trinomial $\widetilde{h}$ with roots $\zeta_1= r$ and $\zeta_2= re^{\ii \theta}$, where $\theta\in (0,2\pi)$. 
Without loss of generality, we assume that $h=\widetilde{h}$.
Since 
$h(\zeta_1)=h(\zeta_2)=0$, we have
\begin{equation}
A=-\frac{C}{r^{n+m}} \frac{e^{\ii m\theta}-1}{e^{\ii(n+2m)\theta}-1}\quad \textrm{ and }\quad
B= -\frac{C}{r^m} \frac{e^{\ii m\theta} (e^{\ii(n+m)\theta}-1)}{e^{\ii(n+2m)\theta}-1}.
\end{equation}
Recall the identity
\begin{equation}
e^{\ii t}-1= 2\ii\cdot \sin\left(\frac{t}{2}\right)e^{\ii\frac{t}{2}}\quad \textrm{ for any }\quad t\in \mathbb{R}.
\end{equation}
Then we have
\begin{equation}
\frac{A^m B^{n+m}}{C^{n+2m}}= \frac{1}{r^{2m(n+m)}}(-1)^{n}\frac{\sin ^m(\frac{m\theta}{2}) \sin^{n+m}(\frac{(n+m)\theta}{2})}{\sin^{n+2m}(\frac{(n+2m)\theta}{2})},
\end{equation}
which with the help of~(iii) of 
Corollary~\ref{cor:real} yields the statement.
\qedw

\subsection{\textbf{Proof of Corollary~\ref{cor:leasttwo}}}\label{sec:cor:leasttwo}
\hfill
\smallskip

\noindent
After a rotation,
without loss of generality, we can assume that $h$ has a real root $r>0$ with multiplicity at least two. 
The function $\gamma(x):= Ax^{n+m} +Bx^m +C$, $x\in \mathbb{R}$ represents a curve in the complex plane $\mathbb{C}$.
Since $r$ is a root of $h$ with multiplicity at least two and   $h(x)= \gamma(x)$ for all $x\in \mathbb{R}$, we have $\gamma(r)=\gamma'(r)=0$, where $\gamma'$ denotes the derivative of $\gamma$.
The latter reads as follows
\begin{equation}
Ar^{n+m}+Br^m+C=0\quad \textrm{ and }\quad
(n+m)Ar^{n+m-1}+mBr^{m-1}=0,
\end{equation}
so $A= \frac{m C}{r^{n+m}n}$, $B= -\frac{(n+m)C}{r^m n}$. A straightforward computation yields
\[\frac{A^mB^{n+m}}{C^{n+2m}} = \frac{(-1)^{n+m}}{r^{2m(n+m)}}\frac{m^m(n+m)^{n+m}}{n^{n+2m}},
\]
which with the help of~(iii) of 
Corollary~\ref{cor:real} yields the statement.
\qedw

\subsection{\textbf{Proof of Theorem~\ref{lem:geometric}}}\label{sec:lem:geometric}
\hfill
\smallskip

\noindent
By~(ii) of Corollary~\ref{cor:real},  
it is enough to show~\eqref{eq:reales}.
By hypothesis we have that $h(r)=0$. 
We assume that $n$ and $m$ are co-prime numbers. 
Since $|A|r^{n+m}$, $|B|r^{m}$ and $|C|$ are the side lengths of some degenerate triangle,
the contrapositive of Lemma~2.11 in~\cite{Barrera} applied to 
$\widetilde{h}(z):=e^{-\ii\gamma}h(z)$, $z\in \mathbb{C}$
 yields
\begin{equation*}
(n+m)(\beta-\gamma)+m(\alpha-\gamma)\equiv 0 \quad \mop.
\end{equation*} 
The latter implies~\eqref{eq:reales}.
Moreover, by~\eqref{eq:hequ} we have that $h$ is Egerv\'ary equivalent to $g_{u,v}$ for some $u,v\in \{-1,1\}$, where $g_{u,v}$ is defined in~\eqref{eq:guv}.

We continue with the proof when $d:=\textsf{gcd}(n+m,m)\in \{2,\ldots,m\}$. Observe that $\textsf{gcd}(n,m)=d$.
Let $n':=n/d$ and $m':=m/d$ and note that 
$\textsf{gcd}(n',m')=1$.
Since $h$ has a root of modulus $r$,
we have that the harmonic trinomial
\begin{equation}
H(z):=A_1z^{n'+m'}+B_1\overline{z}^{m'}+C_1\quad \textrm{ for all }\quad z\in \mathbb{C}
\end{equation}
has a root of modulus $r^d$. Then the previous discussion for the co-prime case implies
\begin{equation*}
(n'+m')(\beta-\gamma)+m'(\alpha-\gamma)\equiv 0 \quad \mop.
\end{equation*}
Multiplying by $d$ in both sides the preceding inequality yields 
\begin{equation*}
(n+m)(\beta-\gamma)+m(\alpha-\gamma)\equiv 0 \quad \mop.
\end{equation*}
In addition, $H$ is Egerv\'ary equivalent to \begin{equation}
\widetilde{g}_{u,v}(z):= u|A|z^{n'+m'}+v|B|\overline{z}^{m'}+|C|,\quad  z\in \mathbb{C},
\end{equation} for some $u,v\in \{-1,1\}$.
The change of variable $z \mapsto z^d$ yields that $h$ is Egerv\'ary equivalent to $g_{u,v}$ for some $u,v\in \{-1,1\}$.

In the sequel, we show~(a). 
We start with the following observation. The relation $|C|=|A|r^{n+m}+|B|r^m$ holds true if and only if  $g_{-1,-1}(r)= 0$. 
By Descartes' rule of signs we have that $r$ is the unique positive real number satisfying $g_{-1,-1}(r)=0$.
By Theorem~\ref{th:egervaryI} one can verify that
\begin{equation}\label{eq:casoseger}
g_{-1,-1}\textrm{ is Egerv\'ary equivalent to }
\begin{cases}
g_{-1,1} & \textrm{ if and only if } n+m \textrm{ is an  even number},\\
g_{1,1} & \textrm{ if and only if } n \textrm{ is an  even number},\\
g_{1,-1} & \textrm{ if and only if } m \textrm{ is an  even number},\\
\end{cases}
\end{equation}
where $g_{u,v}$ is defined in~\eqref{eq:guv}.

Now, we assume that $n$ and $m$ are co-prime numbers.
Then we claim that $g_{-1,1}$, $g_{1,1}$ and $g_{1,-1}$ are never Egerv\'ary equivalent between them.
Indeed,  we start assuming that $n+m$ is an even number.
By~\eqref{eq:casoseger} we have $g_{-1,-1}$ is Egerv\'ary equivalent to $g_{-1,1}$.
Since $n$ and $m$ are co-prime numbers,  the assumption that $n+m$ is an even number imply that $n$ and $m$ are odd numbers.
Recall that being Egerv\'ary equivalent is an equivalence relation. Hence~\eqref{eq:casoseger} yields that $g_{-1,1}$ cannot be Egerv\'ary equivalent to $g_{1,1}$ neither $g_{1,-1}$ when $n+m$ is an even number.

We now claim that $g_{1,1}$ and $g_{1,-1}$  do not have a root of modulus $r$. We start showing that $g_{1,-1}$ does not have a root of modulus $r$.
Indeed, by contradiction 
assume that there exists $\zeta=r e^{\ii \theta}$ with $\theta\in [0,2\pi)$ such that $g_{1,-1}(\zeta)=0$, that is, 
\[
|A|r^{n+m}e^{\ii(\theta(n+m))}+|B|r^me^{-\ii (\theta m+\pi)}+|C|=0.
\]
Since  $|C|= |A|r^{n+m}+|B|r^m$, Lemma~A.3 in Appendix~A of~\cite{Barrera} implies
\begin{equation}
\theta(n+m)\equiv \pi \quad \mo\quad\textrm{ and }\quad 
-\theta m -\pi \equiv \pi \quad \mo.
\end{equation}
Then we have
\[ \theta= \frac{(2k+1)\pi}{n+m}= \frac{2\pi k'}{m}\]
for some $k, k' \in \mathbb{Z}$. Hence, $(2k+1)m= 2k'(n+m)$, which is a contradiction since $m$  and $2k+1$ are odd numbers.

Now, we prove that $g_{1,1} $ has no root of modulus $r$.
By contradiction,
assume that there exists a root of  $g_{1,1} $ of the form $\zeta=r e^{\ii \theta}$ with $\theta\in [0,2\pi)$. Then it follows that
\[
|A|r^{n+m}e^{\ii(\theta(n+m))}+|B|r^me^{-\ii (\theta m)}+|C|=0.
\]
Similarly to the previous case, we obtain 
\begin{equation}
\theta(n+m)\equiv \pi \quad \mo\quad\textrm{ and }\quad 
-\theta m  \equiv \pi \quad \mo,
\end{equation}
which implies $(2k+1)m= (2k'+1)(n+m)$ for some $k, k' \in \mathbb{Z}$. This yields a contradiction since $m$ and $2k+1$ are odd numbers and $n+m$ is an even number. 

We observe that~\eqref{eq:guv} yields that $h$ is Egerv\'ary equivalent to $g_{u,v}$ for some $u,v\in \{-1,1\}$. The preceding analysis implies that $h$ is Egerv\'ary equivalent to $g_{-1,-1}$, which is also Egerv\'ary equivalent to $g_{-1,1}$. 

The proof when $n$ is an even number and the proof when $m$ is an even number follow similarly and we omit them. In summary, the proof of~(a) is complete.
Moreover, the proofs of~(b)~and~(c) are analogous. 
\qedw

\subsection{\textbf{Proof of Theorem~\ref{th:samemodulus}}}\label{sec:th:samemodulus}
\hfill
\smallskip

\noindent
Since~\eqref{eq:cociente} is valid,
without loss of generality we assume that
\[
\left|\frac{A_1}{A_2}\right|=\left|\frac{B_1}{B_2}\right|=\left|\frac{C_1}{C_2}\right|=1,
\]
that is,
$|A_1|= |A_2|$, $|B_1|= |B_2|$ and $|C_1|= |C_2|$.
Let $r>0$ be fixed. Then the following straightforward remark is true: 
$|A_1|r^{n+m}$, $|B_1|r^m$ and $|C_1|$ are the side lengths of a triangle $\Delta_1$ (it may be degenerate), if and only if, $|A_2|r^{n+m}$, $|B_2|r^m$ and $|C_2|$ are the side lengths of a triangle $\Delta_2$.
In fact, $\Delta_1$ and $\Delta_2$ are congruent.

By hypothesis, $h_1$ and $h_2$ have roots (such roots may be different) of modulus $r$ for some $r>0$.
The proof is divided in three cases accordingly to $|A_1|r^{n+m}$, $|B_1|r^m$ and $|C_1|$ are the side lengths of some triangle.

We now assume that $n$ and $m$ are co-prime numbers.

\noindent
{Case~(1).} Assume that $|A_1|r^{n+m}$, $|B_1|r^m$ and $|C_1|$ are not the side lengths of any triangle. By Lemma~2.6 in~\cite{Barrera} we have that there is no root of modulus $r$ for the harmonic trinomial $h_1$ and $h_2$, which yields a contradiction.

\noindent
{Case~(2).} 
Assume that $|A_1|r^{n+m}$, $|B_1|r^m$ and $|C_1|$ are the side lengths of some triangle. 
For each $j\in \{1,2\}$, we set the corresponding pivotals
\begin{equation}\label{eq:pivotals}
\begin{split}
P_{*,j}&= \frac{(n+m)(\beta_j-\gamma_j -\pi)+ m(\alpha_j-\gamma_j -\pi)}{2\pi}
\quad
\textrm{ and }\\
\omega_{*,j}&= \frac{(n+m)w_1-mw_2}{2\pi},
\end{split}
\end{equation}
where $w_1$ and $w_2$ are the angles opposite to the side lengths $|A_1|r^{n+m}$ and $|A_2|r^{m}$, respectively. We note that
$\omega_{*,1}(r)= \omega_{*,2}(r)$. By Proposition~2.3  in~\cite{Barrera} 
  for each $j=1, 2$ we have that 
  $P_{*,j}+ \omega_{*,j}$ or $P_{*,j}- \omega_{*,j}$ are integers numbers.
  If $P_{*,1}+ \omega_{*,1}(r)$ and $P_{*,2}+ \omega_{*,2}(r)$ are integers, then 
$P_{*,1}-P_{*,2}$ is an integer and by~\eqref{eq:pivotals} we deduce
\[
(n+m)(\beta_1-\beta_2)+m(\alpha_1-\alpha_2)-(n+2m)(\gamma_1-\gamma_2)\equiv 0\quad \mo.
\]
The remainder cases
are similar  and hence we omit their proofs.
 
\noindent
{Case~(3).} Assume that $|A_1|r^{n+m}$, $|B_1|r^m$ and $|C_1|$ are the side lengths of some  degenerate triangle. 
By~(a),~(b)~and~(c) of Theorem~\ref{lem:geometric} we have that $h_1$ and $h_2$ are Egerv\'ary equivalent. Hence~\eqref{eq:mod} in Theorem~\ref{th:egervaryI} yields~\eqref{ecn:one}.

By Case~(1), Case~(2) and Case~(3) we finish the proof for the co-prime setting.

We continue with the proof of~\eqref{ecn:one} for $d:=\textsf{gcd}(n+m,m)\in \{2,\ldots,m\}$.
Let $n':=n/d$ and $m':=m/d$ and note that 
$\textsf{gcd}(n',m')=1$.
Since $h_1$ and $h_2$ have roots (such roots may be different) of modulus $r$ for some $r>0$,
we have that the harmonic trinomials
\begin{equation}
H_1(z):=A_1z^{n'+m'}+B_1\overline{z}^{m'}+C_1\quad \textrm{ for all }\quad z\in \mathbb{C}
\end{equation}
and
\begin{equation}
H_2(z):=A_2z^{n'+m'}+B_2\overline{z}^{m'}+C_2\quad \textrm{ for all }\quad z\in \mathbb{C}
\end{equation}
have roots (such roots may be different) of modulus $r^d$. Then the previous discussion for the co-prime case implies
\begin{equation}\label{eq:nada}
m'(\alpha_1\pm \alpha_2)+(n'+m')(\beta_1\pm \beta_2) - (n'+2m')(\gamma_1\pm \gamma_2)\equiv 0 \quad \mo.
\end{equation}
Multiplying by $d$ in both sides of~\eqref{eq:nada} 
gives~\eqref{ecn:one}. Finally, Theorem~\ref{th:egervaryI} yields that $h_1$ and $h_2$ are Egerv\'ary equivalent.
\qedw

\appendix

\section{\textbf{Tools}}\label{append}
This section contains auxiliary results that help us to make this paper more fluid. 

\begin{lemma}[Linear Diophantine solutions I]\label{lem:parame}
\hfill

\noindent
Let $n,m\in \mathbb{N}$ be fixed.
Then
the solutions $x_1$, $x_2$, $x_3$ of the linear Diophantine equation
\begin{equation}\label{eq:po1}
m(x_1-x_3)+(n+m)(x_2-x_3)
\equiv 0 \quad \mo
\end{equation}
can be parametrized  as follows
\begin{equation}\label{eq:po2}
x_1 \equiv x_3 +(n+m)\delta \quad  \mo\quad \textrm{ and }\quad  x_2 \equiv x_3 - m\delta \quad  \mo
\end{equation}
for some $\delta\in \mathbb{R}$. 
\end{lemma}

\begin{proof}
We note that~\eqref{eq:po1} reads as follows
\begin{equation}\label{eq:2pik}
m(x_1-x_3)+(n+m)(x_2-x_3)
=2\pi k\quad \textrm{ for some }\quad k\in \mathbb{Z}.
\end{equation}
Let $k\in \mathbb{Z}$ be fixed.
We observe that the solutions of the homogeneous equation
\begin{equation}
m(x_1-x_3)+(n+m)(x_2-x_3)
=0
\end{equation}
can be parametrized by
$x_1-x_3=(n+m)\delta$
and 
$x_2-x_3=-m\delta$  for $\delta\in \mathbb{R}$.
If $k=0$ we immediately obtain~\eqref{eq:po2}.
Then we assume that $k\in \mathbb{Z}\setminus\{0\}$.
One can see that the solutions of~\eqref{eq:2pik} can be parametrized by 
\begin{equation}\label{eq:ghtp}
x_1-x_3=(n+m)\delta +2\pi z_1
\quad\textrm{ and }\quad
x_2-x_3=-m\delta+2\pi z_2
\quad \textrm{ for } \quad \delta\in \mathbb{R},
\end{equation}
where $z_1\in \mathbb{R}$ and $z_2\in \mathbb{R}$ is a particular solution of the linear Diophantine  equation
$mz_1+(n+m)z_2=k$.
Now, we assume that $\textsf{gcd}(n+m,m)=1$. 
Then B\'ezout's Identity implies that
there exists $z^*_1\in \mathbb{Z}$ and $z^*_2\in \mathbb{Z}$ satisfying 
$mz^*_1+(n+m)z^*_2=k$.
Choosing $z_1=z^*_1$ and $z_2=z^*_2$ 
in~\eqref{eq:ghtp},
we obtain~\eqref{eq:po2}.

We continue with the proof of~\eqref{eq:po2} for $d:=\textsf{gcd}(n+m,m)\in \{2,\ldots,m\}$.
Let $n':=n/d$ and $m':=m/d$ and note that 
$\textsf{gcd}(n',m')=1$.
We rewrite~\eqref{eq:2pik} as follows
$m'(x'_1-x'_3)+(n'+m')(x'_2-x'_3)
=2\pi k$ for some $ k\in \mathbb{Z}$,
where $x'_1=x_1 d$, $x'_2:=x_2 d$ and $x'_3:=x_3 d$.
Since $\textsf{gcd}(n'+m',m')=1$,
the previous reasoning yields
\begin{equation}\label{eq:ec}
x'_1-x'_3=(n'+m')\delta +2\pi z'_1
\quad\textrm{ and }\quad
x'_2-x'_3=-m'\delta+2\pi z'_2
\quad \textrm{ for } \quad \delta\in \mathbb{R},
\end{equation}
and some integers $z'_1$ and $z'_2$. 
Multiplying by $d$ in  both sides of the equalities given in~\eqref{eq:ec} we 
obtain~\eqref{eq:po2}.
\end{proof}

\begin{lemma}[Linear Diophantine solutions II]\label{lem:paramedos}
\hfill

\noindent
Let $n,m\in \mathbb{N}$ be fixed.
Then
the solutions $x_1$, $x_2$, $x_3$ of the linear Diophantine equation
\begin{equation}\label{eq:po1pi}
m(x_1-x_3)+(n+m)(x_2-x_3)
\equiv 0 \quad \mop
\end{equation}
can be parametrized  as follows
\begin{equation}\label{eq:po2pi}
x_1 \equiv x_3 +(n+m)\delta \quad  \mop\quad \textrm{ and }\quad  x_2 \equiv x_3 - m\delta \quad  \mop
\end{equation}
for some $\delta\in \mathbb{R}$. 
\end{lemma}

\begin{proof}
The proof follows step by step from the proof of Lemma~\ref{lem:parame} replacing $2\pi$ by $\pi$.
\end{proof}

\section*{\textbf{Declarations}}
\noindent

\medskip

\noindent
\textbf{Acknowledgments.} 
\hfill

\noindent
G. Barrera would like to express his gratitude to University of Helsinki, Department of Mathematics and
Statistics, for all the facilities used along the realization of this work. 
He thanks the Faculty of Mathematics, UADY, Mexico, for the hospitality during the research visit in 2022, where partial work on this paper was undertaken. All authors are greatly indebted with professor P\'eter Kevei
(Bolyai Institute, University of Szeged) for his support on the translation of~\cite{eger} and professor P\'eter G\'abor Szab\'o (Bolyai Institute, Department of Computational Optimization, University of Szeged) for providing us a copy of the original 
manuscript~\cite{eger}.
The authors
are grateful to the reviewer for the thorough examination of the paper, which has lead to a significant improvement.

\medskip

\noindent
\textbf{Funding.} 
\hfill

\noindent
The research of G. Barrera has been supported by the Academy of Finland,
via an Academy project (project No. 339228) and the Finnish Centre of Excellence in Randomness and STructures (project No. 346306).
The research of W. Barrera and 
J.P. Navarrete has been supported by the CONACYT, ``Proyecto Ciencia de Frontera'' 2019--21100 via Faculty of Mathematics, UADY, M\'exico.

\medskip

\noindent
\textbf{Ethical approval.} 
\hfill

\noindent
Not applicable. 

\medskip

\noindent
\textbf{Competing interests.}
\hfill

\noindent
The authors declare that they have no conflict of interest.

\medskip

\noindent
\textbf{Authors' contributions.}
\hfill

\noindent
All authors have contributed equally to the paper.

\medskip

\noindent
\textbf{Availability of data and materials.}
\hfill

\noindent
Data sharing not applicable to this article as no data-sets were generated or analyzed during the current study.

\end{document}